\documentclass[12pt]{amsart}

\usepackage{graphicx}

\usepackage{subcaption} %para facer a artistada coas figuras nun grid
\usepackage{quiver}
\usepackage{adjustbox}
\usepackage{tikz-cd}
\usepackage{hyperref}
\hypersetup{bookmarks=true,
	unicode=true,
	colorlinks=true,
	citecolor=black,
	linkcolor=black,
	urlcolor=black,
	plainpages=false,
	pdfpagelabels=true}

\usepackage[new]{old-arrows}
\usepackage{comment}
\usepackage{xfrac}
\usepackage{faktor}
\usepackage{graphicx}
\usepackage{resizegather}
\usepackage{float}

 \usepackage{url}	%url adresses
 \allowdisplaybreaks %align may cross pages

\usepackage{pgf}

\usepackage{xcolor}

\usepackage{comment} % para comentar trozos grandes de codigo

\newtheorem{teo}{Theorem}[section]
\newtheorem{theorem}[teo]{Theorem}

\newtheorem{corollary}[teo]{Corollary}

\newtheorem{lemma}[teo]{Lemma}

\newtheorem{proposition}[teo]{Proposition}

\theoremstyle{definition}
\newtheorem{definition}[teo]{Definition}

\newtheorem{example}[teo]{Example}

\theoremstyle{remark}
\newtheorem{remark}{Remark}

\numberwithin{figure}{section}%figure numbering

\newcommand{\cat}{\mathop{\mathrm{cat}}}

\newcommand{\co}{\colon}

\newcommand{\D}{\mathrm{D}}

\newcommand{\F}{\mathbb{F}}

\newcommand{\id}{\mathrm{id}}

\newcommand{\wtilde}{\widetilde}

\newcommand{\HD}{\mathrm{H^{\sbt}\mathrm{D}}}

\newcommand{\sbt}{\,\begin{picture}(-1,1)(-1,-3)\circle*{3}\end{picture}\ }

\newcommand{\HhomD}{\mathrm{H_{\sbt}\mathrm{D}}}

\newcommand{\CP}{\mathbb{C}\mathrm{P}}
\newcommand{\cte}{\mathrm{c}}

\newcommand{\Hcat}{\mathrm{H^{\sbt}cat}}

\newcommand{\Hscat}{\mathrm{H^{\sbt}scat}}
\newcommand{\HsD}{\mathrm{H^{\sbt}s\mathrm{D}}}
\newcommand{\HsTC}{\mathrm{H^{\sbt}sTC}}
\newcommand{\HTC}{\mathrm{H^{\sbt}TC}}

\newcommand{\Hw}{\operatorname{H^{\sbt}w}}
\newcommand{\hw}{\operatorname{hw}}

\newcommand{\im}{\mathrm{im\,}}
\newcommand{\J}{\mathcal{J}}
\newcommand{\lcp}{\mathrm{l.c.p.}\,}

\newcommand{\RP}{\mathbb{R}\mathrm{P}}
\renewcommand{\S}{\mathbb{S}}
\newcommand{\scat}{\mathrm{scat}}

\newcommand{\sD}{\mathrm{s}\mathrm{D}}
\newcommand{\sd}{\mathrm{sd}}

\newcommand{\TC}{\mathrm{TC}}

\newcommand{\Z}{\mathbb{Z}}

\providecommand{\abs}[1]{\lvert#1\rvert}

\newcommand{\Smile}[2]{%
  \mathop{%
    \raisebox{-0.3ex}{\scalebox{1.5}{$\smile$}}%
  }\limits_{#1}^{\raisebox{1.2ex}{\scriptsize$#2$}}%
}

\begin{document}

\title[A Computational (Co)homological Approach to Contiguity]{A Computational (Co)homological Approach to Contiguity Distance}
\thanks{The first author was partially supported by Xunta de Galicia ED431C 2019/10 with FEDER funds. The second author was partially supported by Deputación Provincial da Coruña (Exp. 2025000022903).
}

\author[E. Mac\'ias]{%
	Enrique Mac\'ias-Virg\'os %etc.
}
		
 \address{%
           	E. Mac\'ias-Virg\'os \\
              CITMAga, Departamento de Matem\'aticas, Universidade de Santiago de Compostela, 15782-SPAIN}
               \email{quique.macias@usc.es}

\author[A. Méndez]{%
	Ángel Méndez-Vázquez %etc.
}
		
 \address{%
           	Ángel Méndez-Vázquez \\
              CITMAga, Departamento de Matem\'aticas, Universidade de Santiago de Compostela, 15782-SPAIN}
               \email{angel.mendez.vazquez@rai.usc.es}

\author[D. Mosquera]{%
	David Mosquera-Lois %etc.
}
	
 \address{%
           	David Mosquera-Lois \\
              Departamento de Matemáticas, Universidade de Vigo, SPAIN}
               \email{david.mosquera.lois@uvigo.gal} 

%%==================================%%
%% sample for unstructured abstract %%
%%==================================%%

\begin{abstract}   
We introduce two new algebraic invariants, the \emph{(co)\-ho\-mo\-lo\-gi\-cal distances} between continuous maps, which provide computable lower bounds for the homotopic distance and strictly refine the classical cup-length estimates.  
We then define the \emph{simplicial cohomological distance} between simplicial maps and prove a convergence theorem showing that, after sufficiently many barycentric subdivisions, it recovers the cohomological distance between the corresponding continuous maps.  
Several explicit computations are presented to illustrate the effectiveness of the proposed approach.

\end{abstract}

%\keywords{Small category, Fibration, Lusternik-Schnirelmann category, Topological complexity, Baues-Wirsching Cohomology, \v{S}varc genus, sectional category}

%%\pacs[JEL Classification]{D8, H51}

%%\pacs[MSC Classification]{35A01, 65L10, 65L12, 65L20, 65L70}

\maketitle

%\tableofcontents

\section*{Introduction}

In \cite{MAC-MOSQ}, the notion of \emph{homotopic distance} $\D(f,g)$ between two continuous maps $f,g\colon X\to Y$ was introduced. This invariant is defined as the least integer $n\geq 0$ for which there exists an open cover $\{U_0,\dots,U_n\}$ of $X$ such that the restrictions $f_{\vert U_j}, g_{\vert U_j}\colon U_j\to Y$ are homotopic, $f_{\vert U_j}\simeq g_{\vert U_j}$, for all $j=0,\dots,n$. %If no such integer exists, one sets $\D(f,g)=\infty$.  

The homotopic distance extends several classical invariants from spaces to maps. In particular, it provides a unified framework that generalizes the Lusternik–Schnirelmann category and the topological complexity.  

While conceptually unifying, $\D(f,g)$ is often challenging to compute in practice. This led to two of the authors \cite{MAC-MOSQ} to introduce the \emph{contiguity or simplicial distance} $\sD(\varphi,\psi)$ between two simplicial maps $\varphi,\psi\colon K \to L$. It is the least integer $n\geq 0$ such that there exists a cover of $K$ by subcomplexes $K_0,\dots,K_n$ with the property that $\varphi_{\vert K_j}$ and $\psi_{\vert K_j}$ lie in the same contiguity class, $\varphi_{\vert K_j}\sim\psi_{\vert K_j}$, for every $j=0,\dots,n$. This invariant was further investigated  in \cite{TT}.  

By construction, the homotopic distance is always bounded above by the contiguity distance of every simplicial approximation. However, a major drawback of the contiguity distance is that it does not necessarily converge to the homotopic distance under subdivision, as illustrated in \cite[Example 4.6]{QUIQUE2}.  

In this article, we introduce two auxiliary invariants, the \emph{co\-ho\-mo\-lo\-gi\-cal distance} $\HD(f,g;R)$ and its homological counterpart $\HhomD(f,g;R)$, defined by requiring an open cover of $X$ such that, for each subset, the restrictions induce the same (co)homology morphism. These invariants retain key topological information while exhibiting markedly better computational behavior. 

We further define a simplicial version, the \emph{simplicial co\-ho\-mo\-lo\-gi\-cal distance} and prove that, unlike simplicial distance, it converges to its continuous version: after enough subdivisions, the simplicial cohomological distance recovers the cohomological distance (Theorem \ref{thm:simplicial_approx_cohomology}).  

Beyond convergence, the cohomological distances provide strictly sharper lower bounds than classical cup-length estimates. To show this, we introduce a notion of \emph{cohomological weight} and we use stable cohomology operations to obtain cohomology classes with weight greater than one (Theorem~\ref{thm:superweight_hd}), leading to strict refinements of cup-length bounds. 

We give examples, including lens spaces, where our cohomological bounds match the known homotopic invariants (cf. Subsections \ref{subsec:Lens_spaces} and \ref{subsection_computations}). Taken together, these results yield a principled, convergent framework to compute approximations on $\D(f,g)$.

\subsection*{Organization} Section~\ref{sec:cohomo_distance} introduces $\HD$ and $\HhomD$, compares them and relates them to existing invariants. Section~\ref{sec:cohomo_weights} develops the cohomological weight formalism to show that the lower bound provided by the cohomological distance improves classical cup-length estimates (Subsection \ref{subsec:Lens_spaces}). Section~\ref{sec:simplicial_cohomo_distance} defines the simplicial cohomological distance and proves the Approximation Theorem~\ref{thm:simplicial_approx_cohomology}. It also reports computations on standard triangulations (Subsection~\ref{subsection_computations}).

\subsection*{Conventions} Throughout this article, $R$ denotes a commutative ring with a unit. All topological spaces and simplicial complexes are assumed to be path-connected. All simplicial complexes are assumed to be finite, as we are primarily interested in a computational perspective.

\section{(Co)Homological Distance}\label{sec:cohomo_distance}
We introduce the cohomological and homological analogues of the homotopic distance. These invariants retain essential topological information while being more suitable for algebraic computation and provide better lower bounds than cup-length invariants.

\subsection{The (co)homological distance}
We define the (co)homological distance between two continuous maps $f,g\colon X\to Y$. 
\begin{definition}
    The {\em cohomological (homological) distance} between $f$ and $g$ with coefficients in $R$, denoted by $\HD(f,g;R)$ (resp. $\HhomD(f,g;R)$), is the least integer $n\geq 0$ for which there exists an open cover $\{U_0,\dots,U_n\}$ of $X$ such that $(f_{\vert U_j})^* = (g_{\vert U_j})^*\colon H^m(Y;R)\to H^m(U_j;R)$ (resp. $(f_{\vert U_j})_* = (g_{\vert U_j})_*\colon H_m(U_j;R)\to H_m(Y;R)$) for all $m\geq 0$ and for all $j=0,\dots,n$.
\end{definition}

\begin{remark}
    We always have $\HD(f,g;R)\leq\D(f,g)$ and $\HhomD(f,g;R)\leq\D(f,g)$.
\end{remark}

\begin{example}\label{ejemplo_Hopf_2}
    Let $p\colon\S^3\to\S^2$ be the Hopf fibration and let $\cte \colon\S^3\to\S^2$ be a constant map. Then $\D(p,\cte)=1$, whereas $\HD(p, \cte ;R)=0$ and $\HhomD(p,\cte;R)=0$. Hence, the above inequalities can be strict.
\end{example}

Let $\id\co X\to X$ be the identity map, $\cte \co X\to X$ a constant map and $\pi_1,\pi_2\co X\times X\to X$ the canonical projections. We define the {\em cohomological (resp. homological) LS category} as $\Hcat(X;R)=\HD(\id,\cte;R)$ (resp. $\HhomD(\id,\cte;R)$) and the {\em cohomological (resp. homological) topological complexity}  $\HTC(X;R)=\HD(\pi_1,\pi_2;R)$ (resp. $\HhomD(\pi_1,\pi_2;R)$).

As a consequence, mimicking the approach in \cite{MAC-MOSQ}, for any pair of maps $f,g\colon X\to Y$, we have:
\begin{enumerate}
    \item $\HD(f,g;R)\leq \Hcat(X;R).$
    \item $\HD(f,g;R)\leq \HTC(Y;R).$
\end{enumerate}

\begin{remark}
    The homological Lusternik-Schnirelmann category was introduced by Fox \cite{FOX} and we recover it as a particular case.
\end{remark}

\subsection{Cup-length}

In \cite{MAC-MOSQ}, the authors established the following cohomological lower bound for the homotopic distance. Let $\J(f,g;R)\subset H^{*}(X;R)$ be the image of  
$$f^{*}-g^{*}\colon H^{*}(Y;R)\to H^{*}(X;R).$$

\begin{definition}
    We denote by $\lcp\J(f,g;R)$ the least integer $n\geq 0$ such that every product of $n+1$ elements of $\J(f,g;R)$ vanishes.
\end{definition}

\begin{theorem}\label{teorema_desigualdad_lcpJ_HD}
    It holds that $\lcp\J(f,g;R)\leq \HD(f,g;R)$.
\end{theorem}

\begin{proof}
    Suppose $\HD(f,g;R)=n$, and let $\{U_0,\dots,U_n\}$ be an open cover of $X$ such that $(f_{\vert U_j})^* = (g_{\vert U_j})^*$ for every $j=0,\dots,n$.  
    Consider an arbitrary product $x_0\smile\cdots\smile x_n$ of $n+1$ elements with $x_k\in\J(f,g;R)$ for each $k=0,\dots,n$. For each open set $\iota\colon U=U_k\hookrightarrow X$ in the cover, consider the long exact sequence of the pair $(X,U)$:
    $$
    \begin{tikzcd}[row sep=large,column sep=normal]
        \cdots \arrow[r] & {H^m(X,U;R)} \arrow[r, "j_U"] & H^m(X;R) \arrow[r, "\iota^*"] & H^m(U;R) \arrow[r] & \cdots \\
                         &                             & H^m(Y;R) \arrow[u, "f^*-g^*"] \arrow[ru, "(f_{\vert U})^*-(g_{\vert U})^*"'] & &
    \end{tikzcd}
    $$

    Since $(f_{\vert U})^*-(g_{\vert U})^*=0$, we have $\J(f,g;R)\subseteq \ker \iota^* = \im j_U$.  
    Thus, for each element $x_k$ of the product $x_0\smile\cdots\smile x_n$, there exists an element $\widetilde{x}_k\in H^*(X,U_k;R)$ such that $j_{U_k}(\widetilde{x}_k)=x_k$.  Consider the following commutative diagram:
    $$
    \begin{tikzcd}[row sep=huge,column sep=huge]
        {H^*(X,U_0;R)\otimes \cdots  \otimes   H^*(X,U_n;R)} \arrow[r, "\smile"] \arrow[d, "{j_{U_0}}\otimes\cdots\otimes {j_{U_n}}"] 
        & {H^*(X ,\bigcup_{k=0}^n U_k;R )} \arrow[d, "j_{X}"] \\
        H^*(X;R)\otimes\cdots  \otimes   H^*(X;R) \arrow[r, "\smile"] & H^*(X;R)
    \end{tikzcd}
    $$

    As $\bigcup_{k=0}^n U_k = X$, we have $H^*(X ,\bigcup_{k=0}^n U_k;R )=H^*(X,X;R)=0$. Consequently, 
    \begin{equation*}\begin{split}
    x_0\smile\cdots\smile x_n &= j_{U_0}(\widetilde{x}_0)\smile\cdots\smile j_{U_n}(\widetilde{x}_n)\\
     &= j_{X}(\widetilde{x}_0\smile\cdots\smile\widetilde{x}_n)= j_X(0)=0.
    \end{split}\end{equation*}
    Therefore $\lcp\J(f,g;R)\leq n=\HD(f,g;R)$, as required.
\end{proof}

For the LS category, taking $f=\id_X$ and $g=\cte$ (a constant map), we have
$$\J(\id_X,\cte;R)=H^{>0}(X;R),$$
so that $\lcp\J(\id_X,\cte;R)$ coincides with the usual \emph{cup-length} of $X$, denoted by $\lcp(X;R)$. This shows that the cohomological LS category improves the classical cup-length lower bound:
$$
\lcp(X;R)\leq \Hcat(X;R)\leq \cat(X).
$$

Now let $\F$ be a field. For the topological complexity, we have the {\em zero divisors cup-length} of $X$, denoted by $\lcp(\ker\smile;\F)$, defined as the least integer $n\geq 0$ such that every product $x_0\cdot\ldots\cdot x_n$ of $n+1$ elements in the kernel of the cup product
$$
\smile\colon H^*(X;\F)\otimes H^*(X;\F)\longrightarrow H^*(X;\F)
$$
vanishes. To show $\lcp(\ker\smile,\F)\leq\HTC(X;\F)$ we need the following result.

\begin{theorem}\label{teorema_longitud_ideal}
    It holds that $\lcp\J(f,g;R)=\lcp\langle\J(f,g;R)\rangle$, where $\langle\J(f,g;R)\rangle$ denotes the ideal generated by $\J(f,g;R)$.  
\end{theorem}

\begin{proof}
    The inequality $\lcp\J(f,g;R)\leq\lcp\langle\J(f,g;R)\rangle$ is clear. Let us prove the converse. We show that if there exists $x_1,\ldots,x_n\in\left\langle\J(f,g;R)\right\rangle$ such that $x_1\smile\cdots\smile x_n \neq 0$, then $\lcp\J(f,g;R)\geq n$.
    
    Each $x_i\in\left\langle\J(f,g;R)\right\rangle$ can be written as a finite sum
    $$x_i=\sum_{k}\Bigl(\alpha_{k}\smile\bigl(f^*(y_{k})-g^*(y_{k})\bigr)\Bigr)$$
    for some $y_{k}\in H^*(Y;R)$ and $\alpha_{k}\in H^*(X;R)$. As the cup product is distributive, we deduce
    \begin{equation*}\begin{split}
    x_i\smile x_j &= \left(\sum_{k_i}\Bigl(\alpha_{k_i}\smile\bigl(f^*(y_{k_i})-g^*(y_{k_i})\bigr)\Bigr)\right) \smile \left(\sum_{k_j}\Bigl(\alpha_{k_j}\smile\bigl(f^*(y_{k_j})-g^*(y_{k_j})\bigr)\Bigr)\right)\\
     &=\sum_{k_i, k_j}
  \biggl(
    \Bigl(\alpha_{k_i} \smile \alpha_{k_j}\Bigr)
    \smile
    \Bigl(
      \bigl(f^*(y_{k_i}) - g^*(y_{k_i})\bigr)
      \smile
      \bigl(f^*(y_{k_j}) - g^*(y_{k_j})\bigr)
    \Bigr)
  \biggr).
    \end{split}\end{equation*}
    Since the product
    $$x_1\smile\cdots\smile x_n=\sum_{k_1,\ldots,k_n}\Biggl(\biggl(\Smile{i=1}{n}\alpha_{k_i}\biggr)\smile\biggl(\Smile{i=1}{n}\bigl(f^*(y_{k_i})-g^*(y_{k_i})\bigr)\biggr)\Biggr)$$
    is nonzero, there exists at least one tuple $(k_1,\ldots,k_n)$ whose term in the summand is nonzero. In particular, for this tuple, the product 
    $$\Smile{i=1}{n}\bigl(f^*(y_{k_i})-g^*(y_{k_i})\bigr)=\bigl(f^*(y_{k_1})-g^*(y_{k_1})\bigr)\smile\cdots\smile\bigl(f^*(y_{k_n})-g^*(y_{k_n})\bigr)$$
    is nonzero, and therefore we conclude $\lcp\J(f,g;R)\geq n$.
\end{proof}

\begin{corollary}
    It holds that $\lcp\J(\pi_1,\pi_2;\F)=\lcp(\ker\smile;\F)$.
\end{corollary}

\begin{proof}
    By the Künneth formula, if $z$ is an element of $\J(\pi_1,\pi_2;\F)$, then $z=x\otimes 1-1\otimes x$ for some $x\in H^*(X;\F)$, so $z\in\ker\smile$.
    
    Conversely, if $z=\sum_i a_i\otimes b_i\in\ker\smile$, we have
    \begin{equation*}\begin{split}
         z = \sum_i a_i\otimes b_i &= \sum_i a_i(1\otimes b_i-b_i\otimes 1) + (a_i\smile b_i)\otimes 1 \\
         &= \sum_i (a_i(1\otimes b_i-b_i\otimes 1)) + \left(\sum_i (a_i\smile b_i)\right)\otimes 1 \\
         &= \sum_i a_i(1\otimes b_i-b_i\otimes 1) \in \left\langle \J(\pi_1, \pi_2)\right\rangle.
    \end{split}\end{equation*}
    Therefore, $\left\langle \J(\pi_1, \pi_2)\right\rangle = \ker\smile$ and the results follows from Proposition \ref{teorema_longitud_ideal}.
\end{proof}

This shows that the cohomological topological complexity improves the classical zero divisors cup-length lower bound:
$$
\lcp(\ker\smile;\F)\leq \HTC(X;\F)\leq \TC(X).
$$

\subsection{Comparison between cohomological and homological distances}
We compare the homological and cohomological versions of the distance. We prove that they coincide when coefficients are taken in a field, and we exhibit examples showing that no general inequality holds between them.

\begin{theorem}
    Let $\F$ be a field and $X$ a topological space with finitely generated cohomology in each degree. Then
    $$
    \HhomD(f,g;\F)=\HD(f,g;\F).
    $$
\end{theorem}

\begin{proof}
    Let $U\subseteq X$ be an open subset. Consider the following commutative diagram obtained from the Universal Coefficient Theorem: %\cite[Theorem 53.5]{MUNKRES}):
    $$
    \begin{tikzcd}[row sep=huge,column sep=huge]
    H^n(Y;\F) \arrow[r, "(f_{\vert U})^*", shift left=1.5] \arrow[d, "\cong"'] 
               \arrow[r, "(g_{\vert U})^*"', shift right=1.5] 
               & H^n(U;\F) \arrow[d, "\cong"] \\
    H_n(Y;\F) 
              & H_n(U;\F) \arrow[l, "(f_{\vert U})_*"', shift right=1.5] 
                          \arrow[l, "(g_{\vert U})_*", shift left=1.5]
    \end{tikzcd}
    $$
    %We have $(f_{\vert U})^*=(g_{\vert U})^*$ if and only if $(f_{\vert U})^*-(g_{\vert U})^*$ is the zero map. This is equivalent to requiring that $\ev\bigl((f_{\vert U})^*-(g_{\vert U})^*\bigr)={(f_{\vert U})_*}^{\vee}-{(g_{\vert U})_*}^{\vee}$ is the zero map. Now, fixing bases $B$ and $B'$ of $H_n(U;\F)$ and $H_n(Y;\F)$, and the corresponding dual bases $\widetilde{B}$ and $\widetilde{B}'$ of $\Hom(H_n(U;\F),\F)$ and $\Hom(H_n(Y;\F),\F)$, the matrix of ${(f_{\vert U})_*}^{\vee}-{(g_{\vert U})_*}^{\vee}$ with respect to $\widetilde{B}'$ and $\widetilde{B}$ is the transpose of the matrix of $(f_{\vert U})_*-(g_{\vert U})_*$ with respect to $B$ and $B'$. Hence $(f_{\vert U})_*-(g_{\vert U})_*$ is the zero map if and only if ${(f_{\vert U})_*}^{\vee}-{(g_{\vert U})_*}^{\vee}$ is zero. 
    We deduce $(f_{\vert U})_*=(g_{\vert U})_*$ if and only if $(f_{\vert U})^*=(g_{\vert U})^*$.
\end{proof}

In general, no direct inequality exists between the homological and the cohomological distance, as illustrated in Example~\ref{ejemplo_desigualdad_hom_cohom_RP2}. The same example also shows that the homological distance improves in some cases the lower bound provided by $\lcp\J(f,g;R)$.

\begin{example}\label{ejemplo_desigualdad_hom_cohom_RP2}
    Let $X=\RP^2$ be the real projective space endowed with the CW-complex structure with one cell in each dimension. Consider the quotient map collapsing the $1$-skeleton to a point,
    $$\pi\colon \RP^2\longrightarrow{\RP^2}/\,{\S^1}\simeq\S^2.$$
    
    The induced homomorphism in integral homology $$\pi_*\colon H_n(\RP^2;\Z)\to H_n(\S^2;\Z)$$ is trivial for every $n>0$. Let $\cte\colon \RP^2\to\S^2$ be a constant map. Then, $\HhomD(\pi,\cte;\Z)=0$. However, since $\pi$ maps the $2$-cell of $\RP^{2}$ onto the $2$-cell of $\S^2$, the induced homomorphism in cohomology
    $$\pi^*\colon H^2(\S^2;\Z)\to H^2(\RP^2;\Z)$$ is the canonical projection $\Z\to\Z_{2}$. Hence $\HD(\pi,\cte;\Z)>0$, thus demonstrating an example in which $\HD(f,g;R)\nleq \HhomD(f,g;R)$. Moreover, in this case we also have $\J(\pi,\cte;\Z)=H^{>0}(\RP^{2};\Z)$,
    so that
    $$
    \lcp\J(\pi,\cte;\Z)=\lcp(\RP^{2};\Z)=1,
    $$
    which shows that $\lcp\J(f,g;R)\nleq \HhomD(f,g;R)$.

    Dually, consider the inclusion of the $1$-skeleton
    $$
    \iota\colon X_{1}\cong \S^1\hookrightarrow \RP^{2}.
    $$
    In this case, the induced map in cohomology 
    $$
    \iota^{*}\colon H^{n}(\RP^{2};\Z)\to H^{n}(\S^{1};\Z)
    $$
    is trivial for every $n>0$, whereas the induced morphism in homology 
    $$
    \iota_{*}\colon H_{1}(\S^{1};\Z)\to H_{1}(\RP^{2};\Z)
    $$
    is the canonical projection $\Z\to\Z_{2}$. If $\cte\colon\S^1\to\RP^2$ denotes a constant map, then $\HD(\iota,\cte;\Z)=0$ while $\HhomD(\iota,\cte;\Z)=1$. This shows that, in general, $\HhomD(f,g;R)\nleq \HD(f,g;R)$. Furthermore, it provides an example in which the homological distance gives a better lower bound than the cup-length invariant since $\lcp\J(\iota,\cte;\Z)=0$.   
\end{example}

\section{Cohomological weights}\label{sec:cohomo_weights}

In this section we show that the cohomological distance $\HD(f,g;R)$ can give strictly better approximations for the homotopic distance $\D(f,g)$ than the cup-length invariant $\J(f,g;R)$.

Inspired by the work of Fadell and Husseini on LS category \cite{FadellHusseini1991} (further investigated by Rudyak \cite{Rudyak}), its generalization to topological complexity by Farber and Grant \cite{FarberGrant2008}, and the homotopic weight defined by two of the authors in \cite{MAC-MOSQ}, we introduce a notion of cohomological weight.

\subsection{The Cohomological Weight}

Let $f,g\colon X\to Y$ be continuous maps between ANR spaces.

\begin{definition}
    We say a nonzero class $u \in H^*(X;R)$ has \emph{cohomological weight} $\Hw(u;R) = k+1$, if $k$ is the greatest integer such that, for any continuous map $\phi\colon A\to X$ satisfying $\HD(f\circ\phi, g\circ\phi; R) \le k$, we have $\phi^*(u) = 0 \in H^*(A;R)$. We set $\Hw(0;R)=\infty$.
\end{definition}

\begin{remark}
    This definition is the cohomological counterpart of the \emph{homotopy weight} $\hw(u)$ defined in \cite[Definition 5.5]{MAC-MOSQ}, which generalizes Rudyak's strict category weight \cite{Rudyak}. Obviously $\hw(u)\ge \Hw(u;R)$.
\end{remark}

From now on we shall assume that our cohomology classes are in $\J(f,g;R)$.

\begin{lemma}\label{lem:minimal_hd}
    If $u \in \J(f,g;R)$, then $1\le \Hw(u; R)$.
\end{lemma}

\begin{proof}
    If $\HD(f\circ\phi, g\circ\phi; R)=0$ then $(f\circ\phi)^* = (g\circ\phi)^*$. Since $u \in \J(f,g;R)$, there exists an element $y \in H^*(Y;R)$ such that $$u = f^*(y) - g^*(y).$$ By applying $\phi^*$, we get $\phi^*(u) = (f\circ\phi)^*(y) - (g\circ\phi)^*(y)=0$.
\end{proof}

\begin{lemma}\label{lem:weight_le_distance}
    If $u \in \J(f,g;R)$ and $u \neq 0$, then $$\Hw(u;R) \le \HD(f,g;R).$$
\end{lemma}

\begin{proof}
    Assume for contradiction that $\Hw(u;R)>\HD(f,g;R)=n$. Then $\HD(f\circ\id_X,g\circ\id_X;R)\leq n$, so $u=\id_X^*(u)=0$, which contradicts the hypothesis $u \neq 0$.
\end{proof}

\begin{proposition}\label{prop:subadd_hd}
    If $u_0, \dots, u_n \in \J(f,g;R)$, then
    $$
    \sum_{k=0}^n \Hw(u_k; R)\le\Hw(u_0 \smile \dots \smile u_n; R).
    $$
\end{proposition}

\begin{proof}
    We will prove it for the case $n=1$. The general result then follows by induction. Let $\Hw(u_0;R)=i+1$ and $\Hw(u_1;R)=j+1$. We want to prove that $\Hw(u_0\smile u_1;R)\ge i+j+2$. Let $\phi\colon A\to X$ be a continuous map such that $\HD(f\circ\phi,g\circ\phi;R)\le i+j+1$. There exists an open cover $\{U_0,\ldots U_{i+j+1}\}$ of $A$ such that $(f\circ\phi_{\vert U_k})^*=(g\circ\phi_{\vert U_k})^*$ for every $k=0,\dots,i+j+1$.

    Now we define $V_0:=U_0\cup\dots\cup U_i$ and $V_1:=U_{i+1}\cup\dots\cup U_{i+j+1}$. Then $\HD(f\circ\phi_{\vert V_0},g\circ\phi_{\vert V_0};R)\le i$ and $\HD(f\circ\phi_{\vert V_1},g\circ\phi_{\vert V_1};R)\le j$ so $\phi^*_{\vert V_0}(u_0)=0$ and $\phi^*_{\vert V_1}(u_1)=0$. Consider the long exact sequence of the pair $(A,V_0)$:
    $$\begin{tikzcd}
    \dots \arrow[r] & {H^m(A,V_0;R)} \arrow[r, "j_{V_0}"] & H^m(A;R) \arrow[r, "\iota_0^*"]                                    & H^m(V_0;R) \arrow[r] & \dots \\
                    &                                     & H^m(X;R) \arrow[u, "\phi^*"] \arrow[ru, "\phi_{\vert V_0}^*"'] &                      &      
    \end{tikzcd}$$
    As $\phi^*_{\vert V_0}(u_0)=\iota_0^*\circ\phi^*(u_0)=0$, there exists an element $\wtilde{u}_0\in H^*(A,V_0;R)$ such that $j_{V_0}(\wtilde{u}_0)=\phi^*(u_0)$. Analogously, $j_{V_1}(\wtilde{u}_1)=\phi^*(u_1)$ for some $\wtilde{u}_1\in H^*(A,V_1;R)$. By the naturality of the cup product, we conclude
    \begin{align*}
       \phi^*(u_0\smile u_1) &= \phi^*(u_0)\smile\phi^*(u_1)=j_{V_0}(\wtilde{u}_0)\smile j_{V_1}(\wtilde{u}_1)\\
        &= j_{V_0\cup V_1}(\wtilde{u}_0\smile \wtilde{u}_1)=0\in H^*(A,V_0\cup V_1;R).\qedhere
    \end{align*}
\end{proof}

\begin{corollary}\label{cor:bound_hd}
    If $u_0, \dots, u_n \in \J(f,g;R)$ and $u_0 \smile \dots \smile u_n \neq 0$, then
    $$\sum_{k=0}^n \Hw(u_k; R)\le \HD(f,g;R).$$
\end{corollary}

\begin{proof}
    Denote $w=u_0 \smile \dots \smile u_n$. By Lemma \ref{lem:weight_le_distance} and Proposition \ref{prop:subadd_hd}, we have $$\sum_{k=0}^n \Hw(u_k; R)\le \Hw(w;R) \le\HD(f,g;R).\qedhere$$
\end{proof}

The key to refining the cup-length is finding classes in $\J(f,g;R)$ with weight $\ge 2$. This can be achieved using cohomology operations, as we show in Theorem \ref{thm:superweight_hd}.

Recall that a \emph{stable cohomology operation} is a natural transformation 
$$\theta \colon H^*(-;R) \longrightarrow H^{*+i}(-;S)$$
that commutes with the suspension isomorphism (see \cite{CohomologyOperations} for more details). The \emph{excess} of $\theta$, denoted by $e(\theta)$, is the largest integer $n>0$ such that $\theta(u)=0$ for every $u\in H^m(X;R)$ with $m<n$.

\begin{theorem}\label{thm:superweight_hd}
Let $\theta \colon H^*(-;R) \to H^{*+i}(-;S)$ be a stable cohomology operation with excess $e(\theta) \ge n > 0$. If $u \in \J(f,g;R) \cap H^n(X;R)$, then its image $\theta(u)$ has weight at least 2:
$$ \Hw(\theta(u); R) \ge 2. $$
\end{theorem}

\begin{proof}
    Let $\phi\colon A \to X$ be a map satisfying $\HD(f\circ\phi, g\circ\phi; R) \le 1$. Therefore, there exists an open cover $A = U_0 \cup U_1$ such that $$(f\circ\phi_{\vert U_j})^* = (g\circ\phi_{\vert U_j})^*$$ for $j=0, 1$. Since $u\in\J(f,g;R)$, we have $u=f^*(y)-g^*(y)$ for some $y \in H^*(Y;R)$. Thus,
    $$ \phi_{\vert U_j}^*(u) = \phi_{\vert U_j}^*(f^*(y)-g^*(y))=(f\circ\phi_{\vert U_j})^*(y)-(g\circ\phi_{\vert U_j})^*(y)=0.$$
    
    Now consider the Mayer-Vietoris exact sequence for the cover $A = U_0 \cup U_1$:
    $$ \cdots \to H^{n-1}(U_0 \cap U_1; R) \xrightarrow{\delta} H^n(A; R) \xrightarrow{i^*} H^n(U_0; R) \oplus H^n(U_1; R) \to \cdots $$
    As $i^*(\phi^*(u)) = 0$, there exists $v \in H^{n-1}(U_0 \cap U_1; R)$ such that $\phi^*(u) = \delta(v)$. By naturality, and since $\theta$ is a stable operation, 
    $$ \phi^*(\theta(u)) = \theta(\phi^*(u)) = \theta(\delta(v))=\delta(\theta(v)). $$
    The class $v$ has degree $\deg(v) = n-1 < e(\theta)$, so $\theta(v) = 0$. Thus we conclude $\phi^*(\theta(u)) = \delta(0) = 0$.
\end{proof}

\subsection{Example of Strict Refinement: Lens spaces}\label{subsec:Lens_spaces}

We now apply this machinery to show that $\Hcat$ and $\HTC$ are strict refinements of their respective cup-length bounds.

Let $L_p:=L(p;1)$ be the 3-dimensional lens space (for $p$ an odd prime). Its $\Z_p$-cohomology algebra is given by $$H^*(L_p; \Z_p) \cong \Lambda(x_1) \otimes \Z_p[y_2] / (y_2^2),$$ with $\abs{x_1}=1$ and $\abs{y_2}=2$. The lower bound $\lcp(L_p; \Z_p) = 2$ is given by the non-zero product $x_1\smile y_2$. The mod $p$ Bockstein homomorphism $$\beta \colon H^{*}(-;\Z_p) \to H^{*+1}(-;\Z_p)$$ is stable with excess $e(\beta) = 1$ and $\beta(x_1) = y_2$ \cite[Example 3E.2]{HATCHER}. From Theorem \ref{thm:superweight_hd} (with $\theta = \beta$ and $u = x_1$), we deduce the image $y_2 = \beta(x_1)$ satisfies $\Hw(y_2;\Z_p) \ge 2$. By applying Corollary \ref{cor:bound_hd}, we obtain:
$$\Hcat(L_p;\Z_p) \ge \Hw(x_1,\Z_p) + \Hw(y_2,\Z_p)\ge 1 + 2 = 3.$$
Since $\cat(L_p)\leq\dim(L_p)=3$, we deduce 
$$\Hcat(L_p;\Z_p)=3>2=\lcp(L_p; \Z_p).$$
We confirm this computationally in Example \ref{lens_space_hcat}, by constructing an explicit $\Z_p$-cohomologically trivial cover for a triangulation of $L_p$ in the case $p=3$.

On the other hand, for the topological complexity $\TC(L_p)$, the lower bound $\lcp(\ker\smile; \Z_p) = 3$ is given by the non-zero product $z_x\cdot z_y\cdot z_y$, where $$z_x = x_1 \otimes 1 - 1 \otimes x_1, \quad z_y = y_2 \otimes 1 - 1 \otimes y_2.$$
As the Bockstein homomorphism is a derivation, 
$$\beta(z_x) = \beta(x_1) \otimes 1 - 1 \otimes \beta(x_1) = y_2 \otimes 1 - 1 \otimes y_2 = z_y.$$
By Theorem \ref{thm:superweight_hd} (with $\theta = \beta$ and $u = z_x$), we deduce the image $z_y = \beta(z_x)$ satisfies $\Hw(z_y;\Z_p) \ge 2$. From Corollary \ref{cor:bound_hd}, we obtain:
$$ \HTC(L_p; \Z_p) \ge \Hw(z_x; \Z_p) + \Hw(z_y; \Z_p) + \Hw(z_y; \Z_p) \ge 1 + 2 + 2 = 5. $$
Since $\TC(L_p)=5$ \cite[Corollary 15]{FarberGrant2008} (notice we use the normalized definition of topological complexity), we conclude
$$\HTC(L_p;\Z_p)=5>3=\lcp(\ker\smile,\Z_p).$$
   
\section{Simplicial cohomological distance}\label{sec:simplicial_cohomo_distance}
We now move to the simplicial setting and introduce a simplicial cohomological distance, which plays the discrete counterpart of the cohomological distance for simplicial maps. Let $K$ and $L$ be finite simplicial complexes and $\varphi,\psi\colon K\to L$ simplicial maps.

\begin{definition}
The {\em simplicial cohomological distance} between $\varphi$ and $\psi$ with coefficients in $R$, denoted by $\HsD(\varphi,\psi;R)$, is the least integer $n\geq 0$ for which there exists a cover $\{K_0,\dots,K_n\}$ of $K$ by $n+1$ subcomplexes such that 
$(\varphi_{\vert K_j})^* = (\psi_{\vert K_j})^*\colon H^m(L;R)\to H^m(K_j;R)$
for every $m\geq 0$ and for each $j=0,\dots,n$.
\end{definition}

We denote $\Hscat(K;R)=\HsD(\cte,\id_K;R)$ and $\HsTC(K;R)=\HsD(\pi_1,\pi_2;R)$ the {\em simplicial cohomological LS category} and the {\em simplicial cohomological topological complexity}, respectively.

The simplicial cohomological distance satisfies properties analogous to those established for the cohomological distance and for the simplicial distance. Moreover, observe that it satisfies: 
$$\HsD(\sd\varphi,\sd\psi;R)\leq\HsD(\varphi,\psi;R).$$
This suggests that subdividing may refine the simplicial distance. Let us show one example where the inequality is strict.

\begin{example}\label{ejemplo_scat_dim}
    Let $K$ be the complete graph $K_5$ considered as a $1$-dimensional simplicial complex. First, observe that $\Hscat(K;R)>1$. Indeed, if we consider a cover of $K$ by two subcomplexes, by the Pigeonhole Principle, one of them has three or more vertices and therefore a cocycle in cohomology. Denote the vertices of $K$ as \{1,2,3,4,5\}. Then the subcomplexes generated by the maximal simplices
    \begin{align*}
        &K_0=[1,2]\cup[1,3]\cup[1,4]\cup[1,5],\\&K_1=[2,3]\cup[2,4]\cup[3,5] \text{ and }\\ &K_2=[2,5]\cup[3,4]\cup[4,5]
    \end{align*} form a cover of $K$ and they guarantee that $\Hscat(K;R)=2$. An analogous argument shows that $\scat(K)=2$ and $\scat(\sd K)=1$ (see \cite[Example 5.6]{QUIQUE2}).
    
    As $\lcp(K;R)=1$, this shows that the inequality $\lcp(K;R)<\Hscat(K;R)$ may be sharp. Furthermore, we have
    $$\Hscat(\sd K;R)=\scat(\sd K)=1<\Hscat(K;R)=\scat(K)=2.$$
\end{example}

\subsection{Convergence}
We now turn our attention to the comparison between the discrete invariants defined on a simplicial complex $K$ and the corresponding continuous invariants of its geometric realization $\abs{K}$. The complex of Figure \ref{figura_triangulo}  shows that the inequality $\D(\abs{\varphi},\abs{\psi})\leq\sD(\sd^n\varphi,\sd^n\psi)$ might be strict for every $n\geq 1$ (see \cite[Example 4.6]{QUIQUE2}).

\begin{figure}[h!]
    \centering
\includegraphics[width=0.3\linewidth]{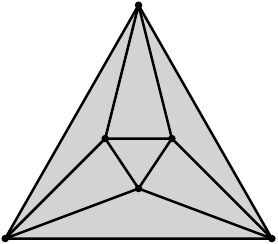}
    \caption{A complex $K$ whose geometric realization is contractible but $\scat(\sd^nK)>0$ for every $n\geq0$}
    \label{figura_triangulo}
\end{figure}

On the contrary, unlike the simplicial distance, we show that the simplicial cohomological distance converges to the cohomological distance of the geometric realization after enough subdivisions.
\begin{theorem}[Approximation theorem]\label{thm:simplicial_approx_cohomology}
    Let $K$ and $L$ be finite simplicial complexes, and let $f,g\colon \abs{K}\to \abs{L}$ be continuous maps such that $\HD(f,g;R)=n$. Then there exists an integer $q\geq 0$, a cover $\{K_0'',\dots,K_n''\}$ of $K''=\sd^q(K)$ by $n+1$ subcomplexes and simplicial maps $\varphi,\psi\colon K''\to L$ whose geometric realizations $\abs{\varphi}$ and $\abs{\psi}$ are homotopic to $f$ and $g$ respectively, satisfying  $(\varphi_{\vert K_j''})^*=(\psi_{\vert K_j''})^*$ for every $j=0,\ldots,n$. In particular, $$\HD(f,g;R)=\HsD(\varphi,\psi;R).$$
\end{theorem}

\begin{proof}
    Suppose $\HD(f,g;R)=n$, and let $\mathcal{U}=\{U_0,\dots,U_n\}$ be a cover of $\abs{K}$ such that $(f_{\vert U_j})^* = (g_{\vert U_j})^*$ for every $j=0,\dots,n$. 
    
    Since $K$ is finite, its geometric realization $\abs{K}$ is compact. Hence, the cover $\mathcal{U}$ admits a Lebesgue number $\delta>0$, i.e., every subset of $\abs{K}$ with diameter less than $\delta$ is contained in some element of $\mathcal{U}$. 
    
    By geometric properties of barycentric subdivision, there exists an integer $m\geq 1$ such that every simplex of the $m$-th barycentric subdivision $K'=\sd^m K$ has diameter less than $\delta$. Consequently, every simplex of $K'$ lies in some $U_j$, so the following subcomplexes of $K'$:
    $$
    K_j' := \bigcup_{\substack{\sigma\in K' \\ |\sigma|\subseteq U_j}} \sigma, 
        \qquad j=0,\dots,n,
    $$
    form a cover of $K'$ by $n+1$ subcomplexes. Moreover, $\abs{K_j'}\subseteq U_j$ for each $j=0,\dots,n$. Hence,
  $(f_{\vert\abs{K_j'}})^*=(g_{\vert\abs{K_j'}})^*$.
    
    By the Simplicial Approximation Theorem, there exists a subdivision $K''=\sd^p K'$ such that $f$ and $g$ admit simplicial approximations $\varphi,\psi\colon K''\to L$. Let $K_j''=\sd^p K_j'$ and consider the cover $\{K_0'',\dots,K_n''\}$ of $K''$ by $n+1$ subcomplexes. Then $(\abs{\varphi}_{\vert\abs{K_j''}})^*=(\abs{\psi}_{\vert\abs{K_j''}})^*.$    
    
    Moreover, by the naturality of the isomorphism between simplicial and singular cohomology, the following diagram commutes:
    $$
    \begin{tikzcd}[row sep=huge,column sep=huge]
    H^*(L;R) \arrow[r, "\cong"] \arrow[d, "(\varphi_{K_j''})^*"', shift right=1.5] 
                          \arrow[d, "(\psi_{K_j''})^*", shift left=1.5] 
    & H^*(\abs{L};R) \arrow[d, "(\abs{\varphi}_{\vert \abs{K_j''}})^*"', shift right=1.5] 
                         \arrow[d, "(\abs{\psi}_{\vert \abs{K_j''}})^*", shift left=1.5] \\
    H^*(K_j'';R) \arrow[r, "\cong"] & H^*(\abs{K_j''};R)
    \end{tikzcd}
    $$
    Consequently, $(\varphi_{\vert K_j''})^* = (\psi_{\vert K_j''})^*$, for all $j=0,\dots,n$ and $q=m+p$.
\end{proof}

\subsection{Computations}\label{subsection_computations}

Using a symbolic computation program implemented in \textsf{SageMath}, we compute several simplicial cohomological distances for specific triangulations which can be found in \cite{LIBRARY}.

\begin{example}
   Let $K$ be the minimal triangulation of the complex projective plane $\CP^2$. The cover $\{K_0,K_1,K_2\}$ of $K$, described in Table \ref{tabla_CP2}, satisfies that for each $j=0,1,2$, the inclusion $\iota\colon K_j\hookrightarrow K$ induces in $\mathbb{Z}_2$-cohomology the same map as a constant simplicial map. Therefore, $$\Hscat(K;\Z_2)=\cat(\CP^2)=2.$$
   \begin{table}[h!]
        \centering
        \caption{Maximal faces of the subcomplexes of the cover of the triangulation of $\CP^2$}
        \label{tabla_CP2}
        \scalebox{0.75}{%
        \begin{tabular}{c|c|c}
            $K_0$ & $K_1$ & $K_2$ \\ \hline
            $[1, 2, 5, 6, 8]$ & $[2, 4, 5, 8, 9]$ & $[2, 3, 4, 6, 9]$ \\
            $[2, 4, 6, 7, 9]$ & $[1, 2, 4, 6, 7]$ & $[1, 2, 3, 7, 9]$ \\
            $[1, 3, 4, 5, 7]$ & $[4, 5, 7, 8, 9]$ & $[2, 3, 6, 7, 9]$ \\
            $[1, 3, 4, 7, 8]$ & $[1, 5, 6, 8, 9]$ & $[1, 2, 4, 5, 9]$ \\
            $[1, 3, 5, 7, 9]$ & $[1, 3, 5, 6, 9]$ & $[2, 3, 4, 8, 9]$ \\
            $[1, 2, 4, 5, 6]$ & $[1, 4, 6, 7, 8]$ & $[3, 4, 5, 7, 8]$ \\
            $[1, 2, 6, 7, 8]$ & $[1, 3, 4, 6, 8]$ & $[2, 5, 6, 7, 8]$ \\
            $[1, 2, 3, 7, 8]$ & $[3, 5, 6, 7, 9]$ & $[1, 2, 3, 8, 9]$ \\
            $[1, 2, 5, 8, 9]$ & $[3, 4, 6, 8, 9]$ & $[2, 3, 5, 7, 8]$ \\
            $[1, 4, 5, 7, 9]$ & $[4, 6, 7, 8, 9]$ & $[2, 3, 4, 5, 6]$ \\
            $[1, 3, 4, 5, 6]$ & $[5, 6, 7, 8, 9]$ & $[2, 3, 5, 6, 7]$ \\
            $[1, 2, 4, 7, 9]$ & $[1, 3, 6, 8, 9]$ & $[2, 3, 4, 5, 8]$
        \end{tabular}}
    \end{table}
\end{example}

\begin{example}
   Let $K$ be the minimal triangulation of the real projective space $\RP^3$. The cover $\{K_0,K_1,K_2,K_3\}$ of $K$, described in Table \ref{tabla_RP3}, satisfies that for each $j=0,1,2,3$, the inclusion $\iota\colon K_j\hookrightarrow K$ induces in $\mathbb{Z}_2$-cohomology the same map as a constant simplicial map. Therefore, $$\Hscat(K;\Z_2)=\cat(\RP^3)=3.$$
   \begin{table}[h!]
    \centering
    \caption{Maximal faces of the subcomplexes of the cover of the triangulation of $\RP^3$}
    \label{tabla_RP3}
    \scalebox{0.75}{%
    \begin{tabular}{c|c|c|c}
        $K_0$ & $K_1$ & $K_2$ & $K_3$ \\ \hline
        $[3, 5, 9, 10]$ & $[1, 4, 8, 10]$ & $[2, 5, 7, 9]$ & $[4, 5, 7, 9]$ \\
        $[3, 6, 7, 8]$ & $[3, 6, 8, 9]$ & $[1, 6, 8, 9]$ & $[1, 3, 7, 10]$ \\
        $[1, 3, 5, 11]$ & $[1, 5, 6, 11]$ & $[2, 4, 6, 11]$ & $[3, 6, 7, 10]$ \\
        $[2, 3, 7, 8]$ & $[3, 4, 8, 9]$ & $[2, 3, 4, 11]$ & $[4, 5, 6, 11]$ \\
        $[3, 4, 5, 11]$ & $[5, 6, 7, 8]$ & $[1, 2, 3, 11]$ & $[1, 4, 7, 10]$ \\
        $[1, 2, 3, 7]$ & $[1, 5, 6, 8]$ & $[1, 2, 6, 11]$ & $[4, 6, 7, 10]$ \\
        $[1, 3, 5, 10]$ & $[2, 5, 7, 8]$ & $[2, 6, 9, 10]$ & $[1, 4, 8, 9]$ \\
        $[3, 4, 5, 9]$ & $[1, 5, 8, 10]$ & $[1, 2, 6, 9]$ & $[1, 2, 7, 9]$ \\
        $[3, 6, 9, 10]$ & $[2, 5, 8, 10]$ & $[2, 4, 6, 10]$ & $[1, 4, 7, 9]$ \\
        $[2, 3, 4, 8]$ & $[2, 4, 8, 10]$ & $[2, 5, 9, 10]$ & $[4, 5, 6, 7]$
    \end{tabular}}
\end{table}
\end{example}

\begin{example}\label{lens_space_hcat}
    Let $K$ be Brehm's vertex-minimal triangulation of the lens space $L_3:=L(3;1)$. The cover $\{K_0,K_1,K_2,K_3\}$ of $K$, described in Table \ref{tabla_L33}, satisfies that for each $j=0,1,2,3$, the inclusion $\iota\colon K_j\hookrightarrow K$ induces in $\mathbb{Z}_3$-cohomology the same map as a constant simplicial map. Therefore, $$\Hscat(K;\Z_3)=\cat(L_3)=3.$$
    \begin{table}[h!]
    \centering
    \caption{Maximal faces of the subcomplexes of the cover of the triangulation of $L_3$}
    \label{tabla_L33}
    \scalebox{0.75}{%
    \begin{tabular}{c|c|c|c}
        $K_0$ & $K_1$ & $K_2$ & $K_3$ \\ \hline
        $[3, 4, 5, 6]$ & $[1, 2, 3, 10]$ & $[2, 6, 7, 11]$ & $[1, 7, 8, 12]$ \\
        $[4, 5, 6, 10]$ & $[3, 4, 6, 7]$ & $[2, 5, 6, 8]$ & $[3, 8, 9, 10]$ \\
        $[2, 7, 10, 11]$ & $[2, 3, 10, 12]$ & $[6, 9, 10, 12]$ & $[1, 8, 11, 12]$ \\
        $[5, 7, 9, 11]$ & $[4, 8, 11, 12]$ & $[3, 9, 10, 12]$ & $[1, 3, 7, 8]$ \\
        $[3, 5, 8, 9]$ & $[1, 2, 4, 9]$ & $[1, 2, 5, 6]$ & $[3, 6, 7, 8]$ \\
        $[1, 2, 10, 11]$ & $[4, 6, 7, 9]$ & $[6, 7, 9, 11]$ & $[2, 4, 8, 9]$ \\
        $[3, 4, 5, 11]$ & $[2, 3, 4, 12]$ & $[2, 6, 7, 8]$ & $[4, 8, 9, 10]$ \\
        $[3, 5, 6, 8]$ & $[1, 2, 5, 9]$ & $[1, 6, 11, 12]$ & $[1, 5, 7, 12]$ \\
        $[5, 6, 10, 12]$ & $[3, 4, 11, 12]$ & $[1, 5, 6, 12]$ & $[1, 3, 8, 10]$ \\
        $[4, 5, 10, 11]$ & $[1, 2, 3, 4]$ & $[1, 2, 6, 11]$ & $[2, 7, 8, 12]$ \\
        $[3, 5, 9, 11]$ & $[1, 3, 4, 7]$ & $[4, 6, 9, 10]$ & $[2, 4, 8, 12]$ \\
        $[5, 7, 10, 11]$ & $[2, 5, 8, 9]$ & $[3, 9, 11, 12]$ & $[4, 8, 10, 11]$ \\
        $[2, 7, 10, 12]$ & $[1, 4, 7, 9]$ & $[6, 9, 11, 12]$ & $[1, 8, 10, 11]$ \\
        $[5, 7, 10, 12]$ & $[1, 5, 7, 9]$ & &
    \end{tabular}}
\end{table}
\end{example}

\begin{example}\label{ejemplohscat_s1xs2}
    Let $K'=[0,1]\cup[0,2]\cup[1,2]$ and $K''=[0,1,2]\cup[0,1,3]\cup[0,2,3]\cup[1,2,3]$ be triangulations of the circle $\S^1$ and the sphere $\S^2$, respectively. The cover $\{K_0,K_1,K_2\}$ of the product $K=K'\times K''$, described in Table~\ref{tabla_S1xS2}, satisfies that for each $j=0,1,2$, the inclusion $\iota\colon K_j\hookrightarrow K$ induces in $\mathbb{Z}_2$-cohomology the same map as a constant simplicial map. Hence, $$\Hscat(K;\Z_2)=\cat(\S^1\times\S^2)=2.$$
    \begin{table}[h!]
        \centering
        \caption{Maximal faces of the subcomplexes of the cover of the triangulation of $\S^1\times\S^2$}
        \label{tabla_S1xS2}
        \scalebox{0.75}{%
        \begin{tabular}{c|c|c}
            $K_0$ & $K_1$ & $K_2$ \\ \hline
            $[(2,2), (1,2), (1,1), (2,3)]$ & $[(1,3), (0,3), (0,0), (0,1)]$ & $[(0,2), (0,3), (2,3), (0,1)]$ \\
            $[(2,2), (1,0), (2,0), (2,3)]$ & $[(1,0), (1,3), (1,1), (0,0)]$ & $[(1,3), (0,2), (0,3), (0,1)]$ \\
            $[(2,2), (0,2), (0,0), (2,3)]$ & $[(2,2), (1,1), (2,3), (2,1)]$ & $[(0,0), (2,3), (2,1), (0,1)]$ \\
            $[(2,2), (1,0), (1,2), (2,3)]$ & $[(1,0), (2,0), (2,3), (2,1)]$ & $[(2,2), (0,2), (2,3), (0,1)]$ \\
            $[(2,2), (2,0), (0,0), (2,1)]$ & $[(1,2), (1,3), (1,1), (0,1)]$ & $[(1,0), (1,2), (1,1), (0,0)]$ \\
            $[(2,2), (1,0), (2,0), (2,1)]$ & $[(1,3), (0,2), (0,3), (0,0)]$ & $[(2,2), (2,3), (2,1), (0,1)]$ \\
            $[(0,2), (0,3), (0,0), (2,3)]$ & $[(1,0), (1,3), (1,1), (2,3)]$ & $[(2,2), (0,0), (2,1), (0,1)]$ \\
            $[(2,2), (2,0), (0,0), (2,3)]$ & $[(1,0), (1,1), (2,3), (2,1)]$ & $[(1,2), (1,1), (0,0), (0,1)]$ \\
            $[(2,0), (0,0), (2,3), (2,1)]$ & $[(1,3), (1,1), (0,0), (0,1)]$ & $[(1,2), (1,3), (0,2), (0,0)]$ \\
            $[(2,2), (0,2), (0,0), (0,1)]$ & $[(1,0), (1,2), (1,3), (2,3)]$ & $[(1,2), (0,2), (0,0), (0,1)]$ \\
            $[(2,2), (1,0), (1,2), (1,1)]$ & $[(1,2), (1,3), (0,2), (0,1)]$ & $[(1,0), (1,2), (1,3), (0,0)]$ \\
            $[(1,2), (1,3), (1,1), (2,3)]$ & $[(2,2), (1,0), (1,1), (2,1)]$ & $[(0,3), (0,0), (2,3), (0,1)]$
        \end{tabular}}
    \end{table}
\end{example}

\begin{example}\label{ejemplohstc_s2}
  Let $K=[0,1,2]\cup[0,1,3]\cup[0,2,3]\cup[1,2,3]$ be a triangulation of the sphere $\S^2$, and consider the product $K\times K$. The cover $\{K_0,K_1,K_2\}$ of $K\times K$, described in Table~\ref{tabla}, satisfies that for each $j=0,1,2$, the projections $\pi_1,\pi_2\colon K_j\to K$ induce in $\mathbb{Z}_2$-cohomology the same map. In this case, $$\HsTC(K;\Z_2)=\TC(\S^2)=2.$$
  \begin{table}[h!]
        \centering
        \caption{Maximal faces of the subcomplexes of the cover of the triangulation of $\S^2\times\S^2$}
        \label{tabla}
        \resizebox{\textwidth}{!}{
        \begin{tabular}{ccc}
        \multicolumn{1}{c|}{$K_0$} & \multicolumn{1}{c|}{$K_1$} & $K_2$ \\ \hline
        \multicolumn{1}{c|}{
        \begin{tabular}[c]{@{}c@{}}
        $[(2,3), (0,1), (0,3), (1,3), (0,0)]$\\
        $[(0,1), (3,3), (3,1), (0,0), (1,1)]$\\
        $[(2,3), (0,2), (3,3), (0,3), (0,0)]$\\
        $[(2,3), (0,2), (0,1), (0,3), (1,3)]$\\
        $[(2,3), (0,1), (1,3), (0,0), (1,1)]$\\
        $[(2,3), (2,1), (0,0), (1,1), (1,0)]$\\
        $[(2,3), (1,2), (0,2), (1,3), (0,0)]$\\
        $[(0,1), (3,3), (3,1), (3,2), (1,1)]$\\
        $[(2,3), (0,2), (0,3), (1,3), (0,0)]$\\
        $[(0,1), (3,3), (3,1), (2,1), (0,0)]$\\
        $[(0,2), (3,3), (2,2), (3,2), (0,0)]$\\
        $[(2,3), (0,2), (0,1), (3,3), (2,2)]$\\
        $[(0,2), (3,3), (0,3), (1,3), (0,0)]$\\
        $[(2,3), (1,2), (0,1), (2,2), (1,1)]$\\
        $[(2,3), (0,1), (3,3), (0,3), (0,0)]$\\
        $[(2,3), (1,2), (0,2), (0,1), (1,3)]$\\
        $[(0,1), (3,3), (1,3), (0,0), (1,1)]$\\
        $[(1,2), (0,2), (0,1), (3,3), (3,2)]$\\
        $[(2,3), (1,2), (0,2), (0,1), (2,2)]$\\
        $[(2,3), (0,2), (0,1), (3,3), (0,3)]$\\
        $[(2,3), (0,2), (3,3), (2,2), (0,0)]$\\
        $[(2,3), (1,2), (0,1), (1,3), (1,1)]$\\
        $[(0,2), (0,1), (3,3), (2,2), (3,2)]$\\
        $[(0,1), (3,3), (2,2), (3,2), (2,1)]$\\
        $[(3,3), (1,3), (0,0), (1,1), (1,0)]$\\
        $[(2,3), (0,1), (2,1), (0,0), (1,1)]$\\
        $[(0,1), (3,3), (3,1), (3,2), (2,1)]$\\
        $[(1,2), (0,1), (3,3), (1,3), (1,1)]$\\
        $[(2,3), (1,2), (0,2), (2,2), (0,0)]$\\
        $[(3,3), (3,1), (0,0), (1,1), (1,0)]$\\
        $[(2,3), (0,1), (2,2), (2,1), (1,1)]$\\
        $[(0,1), (3,3), (0,3), (1,3), (0,0)]$
        \end{tabular}
        } & \multicolumn{1}{c|}{
        \begin{tabular}[c]{@{}c@{}}
        $[(3,3), (2,2), (2,0), (3,2), (0,0)]$\\
        $[(2,2), (2,1), (0,0), (1,1), (1,0)]$\\
        $[(2,2), (2,0), (2,1), (0,0), (1,0)]$\\
        $[(2,3), (3,3), (1,3), (1,1), (1,0)]$\\
        $[(2,3), (3,3), (2,2), (2,0), (0,0)]$\\
        $[(2,3), (2,0), (2,1), (0,0), (1,0)]$\\
        $[(2,3), (1,2), (3,3), (2,2), (1,0)]$\\
        $[(2,3), (3,3), (2,1), (1,1), (1,0)]$\\
        $[(3,3), (2,0), (3,1), (2,1), (0,0)]$\\
        $[(1,2), (0,1), (2,2), (0,0), (1,1)]$\\
        $[(2,3), (0,1), (3,3), (2,1), (0,0)]$\\
        $[(2,3), (1,2), (2,2), (0,0), (1,0)]$\\
        $[(1,2), (0,2), (0,1), (2,2), (0,0)]$\\
        $[(2,3), (1,2), (1,3), (0,0), (1,0)]$\\
        $[(1,2), (2,2), (0,0), (1,1), (1,0)]$\\
        $[(0,2), (0,1), (2,2), (3,2), (0,0)]$\\
        $[(3,3), (3,0), (2,0), (3,1), (1,0)]$\\
        $[(3,3), (3,1), (3,2), (2,1), (1,1)]$\\
        $[(3,3), (3,0), (2,0), (3,2), (1,0)]$\\
        $[(3,3), (2,2), (2,0), (3,2), (1,0)]$\\
        $[(2,3), (3,3), (2,0), (2,1), (1,0)]$\\
        $[(2,3), (1,3), (0,0), (1,1), (1,0)]$\\
        $[(3,3), (3,0), (2,0), (3,1), (0,0)]$\\
        $[(2,3), (1,2), (3,3), (1,3), (1,0)]$\\
        $[(1,2), (3,3), (2,2), (3,2), (1,0)]$\\
        $[(2,3), (2,2), (2,0), (0,0), (1,0)]$\\
        $[(2,3), (3,3), (2,0), (2,1), (0,0)]$\\
        $[(2,3), (3,3), (2,2), (2,0), (1,0)]$\\
        $[(0,1), (2,2), (2,1), (0,0), (1,1)]$\\
        $[(3,3), (3,1), (2,1), (1,1), (1,0)]$\\
        $[(3,3), (2,0), (3,1), (2,1), (1,0)]$\\
        $[(3,3), (3,0), (2,0), (3,2), (0,0)]$
        \end{tabular}
        } & 
        \begin{tabular}[c]{@{}c@{}}
        $[(3,3), (2,2), (2,0), (3,2), (0,0)]$\\
        $[(2,2), (2,1), (0,0), (1,1), (1,0)]$\\
        $[(2,2), (2,0), (2,1), (0,0), (1,0)]$\\
        $[(2,3), (3,3), (1,3), (1,1), (1,0)]$\\
        $[(2,3), (3,3), (2,2), (2,0), (0,0)]$\\
        $[(2,3), (2,0), (2,1), (0,0), (1,0)]$\\
        $[(2,3), (1,2), (3,3), (2,2), (1,0)]$\\
        $[(2,3), (3,3), (2,1), (1,1), (1,0)]$\\
        $[(3,3), (2,0), (3,1), (2,1), (0,0)]$\\
        $[(1,2), (0,1), (2,2), (0,0), (1,1)]$\\
        $[(2,3), (0,1), (3,3), (2,1), (0,0)]$\\
        $[(2,3), (1,2), (2,2), (0,0), (1,0)]$\\
        $[(1,2), (0,2), (0,1), (2,2), (0,0)]$\\
        $[(2,3), (1,2), (1,3), (0,0), (1,0)]$\\
        $[(1,2), (2,2), (0,0), (1,1), (1,0)]$\\
        $[(0,2), (0,1), (2,2), (3,2), (0,0)]$\\
        $[(3,3), (3,0), (2,0), (3,1), (1,0)]$\\
        $[(3,3), (3,1), (3,2), (2,1), (1,1)]$\\
        $[(3,3), (3,0), (2,0), (3,2), (1,0)]$\\
        $[(3,3), (2,2), (2,0), (3,2), (1,0)]$\\
        $[(2,3), (3,3), (2,0), (2,1), (1,0)]$\\
        $[(2,3), (1,3), (0,0), (1,1), (1,0)]$\\
        $[(3,3), (3,0), (2,0), (3,1), (0,0)]$\\
        $[(2,3), (1,2), (3,3), (1,3), (1,0)]$\\
        $[(1,2), (3,3), (2,2), (3,2), (1,0)]$\\
        $[(2,3), (2,2), (2,0), (0,0), (1,0)]$\\
        $[(2,3), (3,3), (2,0), (2,1), (0,0)]$\\
        $[(2,3), (3,3), (2,2), (2,0), (1,0)]$\\
        $[(0,1), (2,2), (2,1), (0,0), (1,1)]$\\
        $[(3,3), (3,1), (2,1), (1,1), (1,0)]$\\
        $[(3,3), (2,0), (3,1), (2,1), (1,0)]$\\
        $[(3,3), (3,0), (2,0), (3,2), (0,0)]$
        \end{tabular}
        \end{tabular}
        }
    \end{table}
\end{example}

\clearpage 

%\bibstyle{sn-mathphys}
\bibliographystyle{plain}
\bibliography{biblio}% common bib file HACER ESTE FICHERO

%% if required, the content of .bbl file can be included here once bbl is generated
%\input sn-article.bbl

%% Default %%
%%\input sn-sample-bib.tex%

\end{document}